\documentclass[a4paper,fleqn]{cas-sc}

\usepackage[numbers, sort&compress]{natbib}
\usepackage{amsmath}
\usepackage{extarrows}

\newtheorem{theorem}{Theorem}
\newtheorem{lemma}[theorem]{Lemma}
\newtheorem{conjecture}[theorem]{Conjecture}
\newtheorem{definition}[theorem]{Definition}
\newtheorem{corollary}[theorem]{Corollary}
\newdefinition{rmk}{Remark}
\newproof{proof}{Proof}
\newproof{pot}{Proof of Conjecture \ref{full}}

\begin{document}
\let\WriteBookmarks\relax
\def\floatpagepagefraction{1}
\def\textpagefraction{.001}

\shorttitle{The Equality Cases For the Laplacian Conjecture of Brouwer}    

\shortauthors{Dongxiu Cai, Zhengbo Chen, Jia Yang, Xiao-Dong Zhang}  

\title [mode = title]{The Equality Cases For the Laplacian Conjecture of Brouwer}

%

\author[1]{Dongxiu Cai}
\ead{diudiutse@sjtu.edu.cn}

\author[1]{Zhengbo Chen}
\ead{czb911@sjtu.edu.cn}

\author[1]{Jia Yang}
\ead{yangjia2020@sjtu.edu.cn}

\author[1]{Xiao-Dong Zhang}
\cormark[1]
\ead{xiaodong@sjtu.edu.cn}

\affiliation[1]{organization={School of Mathematical Sciences, MOE-LSC, SHL-MAC, Shanghai Jiao Tong University},
                addressline={800 Dongchuan Road},
                city={Shanghai},
                postcode={200240},
                country={P.R. China}}

\cortext[cor1]{Corresponding author}

\begin{abstract}
The  Laplacian conjecture of Brouwer  asserts that for any graph \(G\) of order n with \(m\) edges, the sum of the \(k\) largest Laplacian
eigenvalues satisfies \(s_k(G) \le m + \binom{k+1}{2}\) for
$k=1, \ldots, n$.  Later, Li and Guo in 2022  further proposed the full Brouwer's Laplacian spectrum conjecture. Recently,
Kothari and Tudose in 2026 proved the Brouwer's conjecture.  Motivated  by their perfect proof and methods, we proved that for a simple graph of order $n$ with $m$ edges and $1\le k\le n-1$, \(s_k(G) = m + \binom{k+1}{2}\) if and only if $G$ is a threshold
graph with clique number \(k+1\), which confirms the full Brouwer
conjecture  proposed by Li and Guo.
\end{abstract}


\begin{highlights}
\item Complete resolution of the equality case of Brouwer's Laplacian conjecture
\item Novel proof via projection reduction and structural forcing
\item Extremal graphs are precisely the threshold graphs with certain clique number
\end{highlights}

\maketitle
\section{Introduction}

Throughout this paper, all graphs are finite, simple and undirected.
Let \(G\) be a  simple  graph with vertex set \(V(G)\) and edge set \(E(G)\), denoted by
\(
n=|V(G)|
\)
and size
\(
m=e(G)=|E(G)|.
\) For $v\in V(G)$, denote by $N(v)$ the set of neighbors of $v$. Let \(L(G)\) be the  Laplacian matrix of \(G\). 
Since \(L(G)\) is positive semidefinite, its eigenvalues are real and nonnegative. We write them in non-increasing order as
\(
\lambda_1(G)\geq \lambda_2(G)\geq \cdots \geq \lambda_n(G)=0.
\)
For \(1\leq k\leq n\), put
\(
s_k(G):=\sum_{i=1}^k \lambda_i(G).
\)
When no confusion may arise, we simply write \(\lambda_i\) and \(s_k\). Denote by \(\mathbf 1\) the all-one column vector of appropriate size, $I_n$ the identity matrix of size $n$, and $J_n$ the all-one matrix of size $n$.

 Let
\(
d_1\geq d_2\geq \cdots \geq d_n
\)
be the degree sequence of \(G\). Denote by its conjugate degree sequence  $d_1^*\ge\ldots\ge d_n^*$ with
\(
d_i^*(G):=\left|\{j:d_j\geq i\}\right|
\). The Laplacian spectrum is closely related to the structure of a graph. One classical example is Kirchhoff's matrix-tree theorem, which expresses the number of spanning trees of a connected graph in terms of the nonzero Laplacian eigenvalues. Another fundamental result concerns the relation between the Laplacian spectrum and the degree sequence. The Grone--Merris conjecture, proved by Bai and now known as the Grone--Merris--Bai theorem, is stated as follows.
\begin{theorem}[\cite{Grone1994,Bai2011}]
Let $G$ be a simple graph with the Laplacian eigenvalues $\lambda_1\ge\ldots\lambda_n$, the degree sequence $d_1\geq  \cdots \geq d_n$ and the conjugate degree sequence $ d_1^*\ge\ldots\ge d_n^*$. Then
\begin{equation}\label{gm1}
(d_1(G), \ldots, d_n(G))\preceq(\lambda_1(G),\ldots,\lambda_n(G))\preceq (d_1^*(G),\ldots,d_n^*(G)).
\end{equation}
Equivalently,
\[
\sum_{i=1}^k d_i(G)\leq \sum_{i=1}^k \lambda_i(G)\leq \sum_{i=1}^k d_i^*(G),\ \ \
\qquad  \text{for } \ k=1, \ldots, n.
\]
\end{theorem}

Motivated by the Grone--Merris-Bai Theorem, Brouwer and Haemers \cite{Brouwer2012} proposed the following upper bound for partial sums of Laplacian eigenvalues.

\begin{conjecture}[\cite{Berndsen2012}] 
\label{blc}
Let \(G\) be a graph with \(n\) vertices and \(m\) edges. Then
\begin{equation}\label{BrC}
s_k(G)\leq m+\binom{k+1}{2}  \qquad  \text{for } \ k=1, \ldots, n.
\end{equation}
\end{conjecture}

This conjecture has attracted considerable attention in spectral graph theory. The cases \(k=1\), \(k=n-1\) and \(k=n\) are immediate or classical, and the case \(k=2\) was proved in \cite{Haemers2010}. The conjecture has also been verified for several important graph classes,
including trees, unicyclic graphs, bicyclic graphs, threshold graphs, split graphs, cographs and regular graphs; see, for example,
\cite{Haemers2010,DuZhou2012,Berndsen2012}.
Further progress has been obtained under additional structural assumptions, including restrictions involving girth, clique number, vertex cover number, diameter, pendant vertices and arboricity; see
\cite{Chen2018,Chen2019,Cooper2021,Ganie2020}. Lin and Wang \cite{Lin2025} established preservation results for Brouwer's conjecture under several edge-addition operations, showing that the conjecture is inherited by various graphs constructed from two graphs that already satisfy it. Torres and Trevisan \cite{Torres2026} introduced the Brouwer Critical Index, showing that for every graph it suffices to verify Brouwer's inequality at a single graph-dependent index in order to obtain the conjecture for all $k$. Approximate versions and related inequalities were recently developed by Lew, including near-quadratic Brouwer-type bounds obtained via partition-density and star-arboricity methods, as well as bounds involving degree sums and matching numbers; see
\cite{Lew2025,Lew2026,Lew2026Partition}. Recently, Wang, Lin, Zhang and Ye \cite{Wang2026} proved the conjecture for \(k=3\), which also implies the corresponding case \(k=n-4\) via a known complement-type reduction. Most recently, Kothari and Tudose \cite{Kothari2026} gave a proof of Brouwer's conjecture using the Grone--Merris--Bai theorem for split graphs, and also established an equivalence between Brouwer's conjecture and the Grone--Merris--Bai theorem.

Besides the validity of Brouwer's inequality itself, it is natural to ask when equality can occur. Li and Guo \cite{Li2022} proposed a conjectural characterization of all extremal graphs. To state it, we first recall the following family of threshold graphs.

\begin{definition}[\cite{Chen2026}]\label{def:Gkrs}
The family \(\mathcal{G}_{k,r,s}\) consists of all graphs \(G\) whose vertex set
admits a partition
\[
    V(G)=U\sqcup R\sqcup S,
\]
where \(U\) is a clique of order \(k\), \(R\) and \(S\) are independent sets with \(|R|=r\geq 1\),
\(|S|=s\geq 0\). Moreover, every vertex in \(R\) is adjacent to every vertex of
\(U\). The vertices \(v_1,\ldots,v_s\) of \(S\) satisfy
\[
N(v_i) \subsetneq V(K_k), \qquad
N(v_{i+1}) \subseteq N(v_i) \quad (1\leq i < s).
\]This family is exactly the class of threshold
graphs with clique number \(k+1\), as observed in \cite{Chen2026}.
\end{definition}
Li and Guo proved that every graph in this family is extremal for Brouwer's inequality.
\begin{theorem}[\cite{Li2022}, Theorem~2.2]\label{thm:LiGuo-split}
For \(k\geq 1\), \(r\geq 1\), \(s\geq 0\), if \(G\in\mathcal{G}_{k,r,s}\), then
\[
s_k(G) = e(G) + \binom{k+1}{2}.
\]
\end{theorem}
They further conjectured that these are the only equality cases.
\begin{conjecture}[\cite{Li2022}]
\label{full}
Let \(G\) be a graph of order \(n\).  For $1\leq k\leq n$,
\begin{equation}\label{equforBC}
s_k(G) = e(G) + \binom{k+1}{2} 
\end{equation}
if and only if \(G\in\mathcal{G}_{k,r,s}\) for some
\(r\geq 1\) and \(s\geq 0\); i.e., a threshold graph with clique number
\(k+1\).
\end{conjecture}
Li and Guo \cite{Li2022} verified Conjecture~\ref{full} for \(k=2\), and also checked it by computer for all graphs with at most \(9\) vertices.

Our main goal is to prove Conjecture~\ref{full}. The rest of
the paper is organized as follows. In Section~2, we introduce the necessary notation, review several results of Berndsen \cite{Berndsen2012} concerning the gap function \(f_t(G)\), and recall Bai's lemma on complement graphs. In Section~3, we revisit the argument of Kothari and Tudose \cite{Kothari2026}, with particular attention to the equality conditions, and show that every extremal graph for Brouwer's conjecture must be a split graph. Section~4 is devoted to split graphs: we strengthen Lemma~6 of Bai \cite{Bai2011}, analyze its equality case, and characterize the extremal split graphs. Finally, in Section~5, we combine these ingredients to prove Conjecture~\ref{full}.

\section{Preliminaries}
\subsection{Notation}
We collect here the notation that will be used throughout the paper.

\medskip
\noindent\textbf{General notation.}
\begin{itemize}
\item For a positive integer $r$, $[r]:=\{1,\dots,r\}$.
\item $e_i$ denotes the column vector whose $i$-th entry is $1$ and all other entries are $0$; its dimension will be clear from the context.
\item For a real number $x$, $x_+:=\max\{x,0\}$ (the positive part of $x$).
\item $\operatorname{tr}(A)$ is the trace of a square matrix $A$.
\item An $n\times n$ real matrix $P$ is called an orthogonal projection matrix if $P^2=P$, $P^T=P$.
\item For real numbers \(d_1,\ldots,d_n\), \(\operatorname{diag}(d_1,\ldots,d_n)\) denotes the \(n\times n\) diagonal matrix whose \(i\)-th diagonal entry is \(d_i\).
\end{itemize}

\medskip
\noindent\textbf{Graphs.}
All graphs are finite, simple and undirected.  For a graph $G$ we use the following notation.
\begin{itemize}
\item $\overline G$ is the complement of $G$.
\item $G[S]$ is the subgraph of $G$ induced by the vertex set $S$.
\item For two disjoint vertex subsets $A,B$, $e_G(A,B)$ is the number of edges between $A$ and $B$.  When the graph is clear from the context we simply write $e(A,B)$.  Moreover, if $A=\{x\}$ we abbreviate $e(\{x\},B)$ as $e(x,B)$.
\item For a vertex $y$ and a vertex subset $A$, $N_A(y)$ denotes the set of neighbours of $y$ inside $A$; again the graph is understood from the context.
\item A \emph{split graph} is a graph whose vertex set can be partitioned into a clique $C$ and an independent set $I$. An independent set is called maximal if no vertex outside it can be added while preserving independence; in the above partition, $I$ is maximal if every vertex of 
$C$ has at least one neighbor in $I$.
\end{itemize}

\medskip
\noindent\textbf{Quantities.}
For a graph $G$ of order $n$ with $m$ edges,
\begin{itemize}
\item $\lambda_1(G)\ge\cdots\ge\lambda_n(G)=0$ are the Laplacian eigenvalues of $G$.
\item For $1\le t\le n$, $s_t(G)=\sum_{i=1}^t\lambda_i(G)$ is the sum of the $t$ largest Laplacian eigenvalues.
\item $d_1^*(G)\ge\cdots\ge d_n^*(G)=0$ is the conjugate degree sequence of $G$, and $D_t(G)=\sum_{i=1}^t d_i^*(G)$.
\item $T(G):=\max\{\,t:0\le t\le n,\ d_t\ge t\,\}$ for non-increasing degree sequence $d_1\ge d_2\ldots\ge d_n$ of a graph $G$.
\item $B_t(G)=m+\binom{t+1}{2}$ and $f_t(G)=B_t(G)-D_t(G)$.
\end{itemize}
\subsection{Basic lemmas}
We first recall the analysis of the function \(f_t\).

\begin{lemma}[\cite{Berndsen2012}, Lemmas~1.5 and 1.6]
\label{lem:berndsen-basic}
Let \(G\) be a graph with non-increasing degree sequence
$
d_1\ge d_2\ge\cdots\ge d_n
$
and conjugate degree sequence \(d^*=(d_1^*,\ldots,d_n^*)\).

\begin{enumerate}
\item[(i)] For \(t=1,\ldots,n\),
\[
d_t\ge t
\quad\Longleftrightarrow\quad
d_t^*\ge t.
\]

\item[(ii)] 

$f_t(G)-f_{t-1}(G)=t-d_t^*(G)$  and the sequence \(f_0(G),f_1(G),\ldots,f_n(G)\) attains its minimum at \(T(G)\).

\end{enumerate}
\end{lemma}
\begin{lemma}[\cite{Berndsen2012}, Lemma~1.7]
\label{lem:berndsen-split}
Let \(G\) be a split graph with at least one edge. If
$
V(G)=C\sqcup I,
$ such that \(C\) is a clique of $N$ and  \(I\) is a maximal independent set, then
\[
T(G)=N,\;\ f_N(G)=0.
\]
Consequently,
\[
D_N(G)=B_N(G)=e(G)+\binom{N+1}{2}.
\]
\end{lemma}
We shall also need the following immediate refinement of Berndsen's argument.
\begin{lemma}
\label{lem:zeros-of-f}
Let \(G\), \(C\), \(I\), and \(N=|C|\) be as in
Lemma~\ref{lem:berndsen-split}. If \(f_t(G)=0\), then either $t=N$, or $t=N-1$
{and} $d_N^*(G)=N$.

\end{lemma}

\begin{proof}
For \(j<N\), all \(N\) vertices of \(C\) have degree at least \(N\), and hence
\[
d_j^*(G)\ge d_N^*(G)\geq N>j.
\]
Therefore
\[
f_j-f_{j-1}=j-d_j^*<0
\qquad (j<N).
\]
For \(j>N=T(G)\), from the definition of $T(G)$,
\[
f_j-f_{j-1}=j-d_j^*>0
\qquad (j>N).
\]
Since \(f_N=0\), the only possible additional zero is \(f_{N-1}\). Moreover,
\[
f_N-f_{N-1}=N-d_N^*,
\]
so
\[
f_{N-1}=0
\quad\Longleftrightarrow\quad
d_N^*=N.
\]
\end{proof}
\begin{lemma}[\cite{Bai2011}, Proposition~12]
\label{lem:complement-GMB}
Let \(G\) be a graph on \(n\) vertices. With the convention
\(s_0(G)=D_0(G)=0\), for \(0\le k\le n-1\),
\[
s_k(G)=D_k(G)\;\Longleftrightarrow s_{n-k-1}(\overline G)=D_{n-k-1}(\overline G).
\]
\end{lemma}
\section{Reduction to split graphs}
\subsection{Projection reduction setup}
Throughout this section, let $P=(P_{ij})$ be an $n\times n$ orthogonal projection matrix of rank $k$, where \(1 \leq k\leq n-1\) and $P\mathbf 1=\mathbf{0}$. Then $\text{tr}(P)=k$. For \(i\neq j\), define an $n\times n$ matrix $M=(M_{ij})$ and an $n\times 1$ vector $v=(v_i)$ as follows:
\begin{equation}\label{Mij}
M_{ij}=P_{ii}+P_{jj}-2P_{ij}-1,
\qquad
v_i=nP_{ii}-k.
\end{equation}

\begin{lemma}[\cite{Kothari2026}]\label{v}
Let $P=(P_{ij})$ be an $n\times n$ orthogonal projection matrix of rank $k$, where \(1 \leq k\leq n-1\) and $P\mathbf 1=\mathbf{0}$. Set $v$ as (\ref{Mij}). Then
\begin{enumerate}
    \item[(i)] $ \  \mathbf1^Tv=\sum_{i=1}^nv_i=0,\ \sum_{i\neq j}|v_i-v_j|^2=2n\|v\|^2,\ \sum_{i\neq j}(P_{ii}-P_{jj})^2=\frac{2}{n}\|v\|^2.$
   
    \item[(ii)]$ \ \ \frac{1}{4}\sum_{i\neq j}[(P_{ii}+P_{jj}-2P_{ij})^2-(P_{ii}-P_{jj})^2]=k(k+1).$
\end{enumerate}
\end{lemma}
\begin{lemma}[\cite{Kothari2026}, Lemma~5.5]\label{5.5}
   Let $P=(P_{ij})$ be an $n\times n$ orthogonal projection matrix of rank $k$, where \(1 \leq k\leq n-1\) and $P\mathbf 1=\mathbf{0}$, $P_{11}\ge P_{22}\ge\ldots\ge P_{nn}$. Set $M$ and $v$ as (\ref{Mij}). For  $r=1, \ldots, n-1$,  let $H_r$ be the split graph on $[n]$ obtained by making $[r]$ a clique, $\{r+1,\cdots,n\}$ an independent set, and adding one edge between $[r]$ and $\{r+1, \ldots, n\}$  across the cut wherever $M_{ij}\geq 0$.   Then
    $$\sum_{i=1}^rv_i+\sum_{i=1}^r\sum_{j=r+1}^n|M_{ij}|=-\text{tr}[L(H_r)-rI_n+\frac{r}{n}J_n]\leq r(n-r).$$
\end{lemma}
\begin{lemma}[\cite{Kothari2026}, Lemma~5.6]\label{+v}
Let  $P=(P_{ij})$ be an orthogonal projection matrix of order $n$ of rank $k$, where \(1 \leq k\leq n-1\) and $P\mathbf 1=\mathbf 0$. Set $M$ and $v$ as (\ref{Mij}). Then \[
\sum_{i\neq j}(P_{ii}+P_{jj}-2P_{ij}-1)_+-k(k+1)=\frac{1}{4}\bigl[\frac{2}{n}\|v\|^2-\sum_{i\neq j}(1-|M_{ij}|)^2\bigr]
\]
\end{lemma}
\begin{proof}By Lemma \ref{v} and the identity
$(x-1)_+=\frac{1}{4}\bigl(x^2-(1-|x-1|)^2\bigr)$,
    \begin{align*}
        \sum_{i\neq j}(P_{ii}+P_{jj}-2P_{ij}-1)_+-k(k+1)
        &=\frac{1}{4}\bigl[\sum_{i\neq j}\bigl((P_{ii}+P_{jj}-2P_{ij})^2-(1-|P_{ii}+P_{jj}-2P_{ij}-1|)^2\bigr)
        \\&\quad -\sum_{i\neq j}\bigl((P_{ii}+P_{jj}-2P_{ij})^2-(P_{ii}-P_{jj})^2\bigr)\bigr]
        \\&=\frac{1}{4}\bigl[\frac{2}{n}\|v\|^2-\sum_{i\neq j}(1-|M_{ij}|)^2\bigr]
    \end{align*}
\end{proof}
\subsection{From equality to split graphs}
\begin{lemma}[\cite{Kothari2026}, Lemmas~5.5 and 5.6]\label{Pijrela}
    Let $P=(P_{ij})$ be an $n\times n$ orthogonal projection matrix of rank $k$, where \(1 \leq k\leq n-1\) and $P\mathbf 1=\mathbf 0$. Then $$\sum_{i\neq j}(P_{ii}+P_{jj}-2P_{ij}-1)_+\leq k(k+1),$$with equality if and only if  \begin{equation}\label{Pijij}
1-|P_{ii}+P_{jj}-2P_{ij}-1| = |P_{ii}-P_{jj}|,\qquad \forall\,i\neq j.
\end{equation}
\end{lemma}
\begin{proof}
Set $M$ and $v$ as (\ref{Mij}). By Lemma \ref{+v}, 
$$\sum_{i\neq j}(P_{ii}+P_{jj}-2P_{ij}-1)_+-k(k+1)=\frac{1}{4}\bigl[\frac{2}{n}\|v\|^2-\sum_{i\neq j}(1-|M_{ij}|)^2\bigr].$$
And after arranging $v_i$ in non-decreasing order (i.e., arranging $P_{ii}$ in non-decreasing order), we have
\begin{align*}
&\quad \sum_{i\neq j}\bigl(1-|M_{ij}|\bigr)^2
\\&\ge\frac{\displaystyle\bigl(\sum_{i\neq j}(1-|M_{ij}|)|v_i-v_j|\bigr)^2}
        {\displaystyle\sum_{i\neq j}|v_i-v_j|^2}
\qquad\text{(Cauchy--Schwarz)} \\[4pt]
&=\frac{2}{n\|v\|^2}
   \Bigl(\sum_{r=1}^{n-1}(v_r-v_{r+1})
         \Bigl[r(n-r)+\operatorname{tr}\bigl[L(H_r)-rI_n+\frac{r}{n}J_n\bigr]
               +\sum_{i=1}^r v_i\Bigr]\Bigr)^2\qquad(\text{Lemma~\ref{5.5}} ) \\[4pt]
&\ge\frac{2}{n\|v\|^2}
   \Bigl(\sum_{r=1}^{n-1}(v_r-v_{r+1})\sum_{i=1}^r v_i\Bigr)^2
\qquad(\text{Lemma~\ref{5.5} and } \sum_{i=1}^r v_i\ge0) \\[4pt]
&=\frac{2}{n}\|v\|^2.
\end{align*}
Thus $\sum_{i\neq j}(P_{ii}+P_{jj}-2P_{ij}-1)_+\leq k(k+1)$, with equality if and only if the above two inequalities become equalities. In other words, after arranging $v_i$ in non-decreasing order, we have
\begin{enumerate}
\item[(i)] there exists a constant $c$ such that for all $i\neq j$,
  \begin{equation}\label{cs}1-|M_{ij}| = c|v_i-v_j|;\end{equation}
\item[(ii)] for every $r$ with $v_r>v_{r+1}$,\begin{equation}\label{layer}\displaystyle\sum_{i=1}^r\sum_{j=r+1}^n(1-|M_{ij}|) = \sum_{i=1}^r v_i.\end{equation}
\end{enumerate}
Substitute (\ref{cs}) into (\ref{layer}), we have \[
c\sum_{i=1}^r\sum_{j=r+1}^n(v_i-v_j) = \sum_{i=1}^r v_i .
\]And we have
$$\begin{aligned}c\sum_{i=1}^r\sum_{j=r+1}^n(v_i-v_j) &=c\bigl[(n-r)\sum_{i=1}^rv_i-r\sum_{j=r+1}^nv_j\bigr]
\\&=c\bigl[(n-r)\sum_{i=1}^rv_i+r\sum_{i=1}^rv_i\bigr] \qquad\text{(since $\sum_{i=1}^nv_i=0$)}
\\&=cn\sum_{i=1}^rv_i.
\end{aligned}$$
So for every $r$ with $v_r>v_{r+1}$ , we have
$$cn\sum_{i=1}^rv_i = \sum_{i=1}^r v_i .$$
Since $v_i$ is in non-decreasing order and $\sum_{i=1}^nv_i=0$, if $\sum_{i=1}^rv_i=0$ for all $r$, then $v=\mathbf 0$. Then from (\ref{cs}) we have $1-|M_{ij}|=0=|P_{ii}-P_{jj}|.$ If $\sum_{i=1}^rv_i\neq 0$ for some $r$, then  $\sum_{i=1}^rv_i\neq 0$ for all $r<n$, and there exists $r_0<n$ such that $v_{r_0}>0\geq v_{r_0+1}$. Therefore, $$cn\sum_{i=1}^{r_0}v_i = \sum_{i=1}^{r_0} v_i \;\Rightarrow c=\frac{1}{n},$$which means \[
1-|P_{ii}+P_{jj}-2P_{ij}-1|=1-|M_{ij}| = |P_{ii}-P_{jj}|,\qquad \forall\,i\neq j.
\]
Conversely, if (\ref{Pijij}) holds, then after arranging $v_i$ in non-decreasing order, (\ref{cs}) and (\ref{layer}) hold, which implies  $$\sum_{i\neq j}(P_{ii}+P_{jj}-2P_{ij}-1)_+= k(k+1).$$
\end{proof}
To see when the Brouwer's inequality becomes equality, we revisit the proof of Conjecture \ref{blc} in \cite{Kothari2026}.
\begin{theorem}[\cite{Kothari2026}]\label{reproduce}
Let $G$ be a graph of order $n$, $V(G)=[n]$. For $1\leq k\leq n-1$,
\begin{equation}\label{Brou}s_k(G)-e(G)-\binom{k+1}{2}\leq 0.\end{equation}
Furthermore, if $s_k(G)-e(G)-\binom{k+1}{2}=0$ for a given $k$, then there exists an $n\times n$ orthogonal projection matrix $P=(P_{ij})$ of rank $k$ with $P\mathbf 1=\mathbf 0$, such that
\[\sum_{ij\in E(G)}(P_{ii}+P_{jj}-2P_{ij}-1)=\binom{k+1}{2}.\]
\end{theorem}
\begin{proof} By the Courant--Fischer min-max principle, there exists an $n\times n$ orthogonal projection matrix $P=(P_{ij})$ of rank $k$ with $P\mathbf 1=\mathbf 0$, such that $s_k(G)=\operatorname{tr}\bigl(P^TL(G)P\bigr)=\operatorname{tr}(PL(G)).$

By Lemma \ref{Pijrela},
    \begin{align*}
        &\quad s_k(G)-e(G)-\binom{k+1}{2}
        \\&=\operatorname{tr}(PL(G))-e(G)-\binom{k+1}{2}
        \\&=\sum_{ij\in E(G)}(P_{ii}+P_{jj}-2P_{ij}-1)-\binom{k+1}{2}
        \\&\leq \frac{1}{2}\sum_{i\neq j}(P_{ii}+P_{jj}-2P_{ij}-1)_+-\binom{k+1}{2}
        \\&\leq 0.
    \end{align*}
And if $s_k(G)-e(G)-\binom{k+1}{2}=0$ for a given \(k\), then \[\sum_{ij\in E(G)}(P_{ii}+P_{jj}-2P_{ij}-1)=\binom{k+1}{2}.\]
\end{proof}

\begin{lemma}\label{Pii}
Let $P=(P_{ij})$ be an $n\times n$ orthogonal projection matrix of rank $k$, where \(1 \leq k\leq n-1\) and $P\mathbf 1=\mathbf 0$. Let $G$ be a graph of order $n$, $V(G)=[n]$. If$$\sum_{ij\in E(G)}(P_{ii}+P_{jj}-2P_{ij}-1)= \binom{k+1}{2},$$then
\begin{enumerate}
    \item[(i)] for $i\in [n]$, $0\leq P_{ii}<1$;
    \item[(ii)] $ij\in E(G)\;\Leftrightarrow P_{ij}=\max(P_{ii},P_{jj})-1\;\Leftrightarrow
P_{ij}<0$;
    \item[(iii)] $ij\notin E(G)\;\Leftrightarrow P_{ij}=\min(P_{ii},P_{jj})\;\Leftrightarrow
P_{ij}\geq0$;
    \item[(iv)] for $i\in [n]$, $P_{ii}=0$ implies that $i$ is an isolated vertex.
\end{enumerate}
\end{lemma}
\begin{proof}
Set $M$ and $v$ as (\ref{Mij}). By Lemma \ref{Pijrela}, we have
\begin{align*}
0=\sum_{ij\in E(G)}(P_{ii}+P_{jj}-2P_{ij}-1)-\binom{k+1}{2}\leq \frac{1}{2}\sum_{i\neq j}(P_{ii}+P_{jj}-2P_{ij}-1)_+-\binom{k+1}{2}\leq 0.
\end{align*}
Thus we have
\[
ij\in E(G)\;\Rightarrow\; M_{ij}\ge 0,\qquad
ij\notin E(G)\;\Rightarrow\; M_{ij}\le 0,
\]
\[\sum_{i\neq j}(P_{ii}+P_{jj}-2P_{ij}-1)_+=k(k+1).\]
Lemma \ref{Pijrela} gives $1-|M_{ij}| = |P_{ii}-P_{jj}|$ for all $i\neq j$.
In the chain below we use the identities
$\displaystyle\max(a,b)=\frac{a+b+|a-b|}{2}$, $\displaystyle\min(a,b)=\frac{a+b-|a-b|}{2}$.

\noindent\textbf{Case 1:} $ij\in E(G)$. Then $M_{ij}\ge 0$, so $|M_{ij}|=M_{ij}$ and
$1-M_{ij}=|P_{ii}-P_{jj}|$. Hence
\[
\begin{aligned}
M_{ij}\ge 0
&\Longleftrightarrow 1-M_{ij} = |P_{ii}-P_{jj}| \\
&\Longleftrightarrow 1-(P_{ii}+P_{jj}-2P_{ij}-1) = |P_{ii}-P_{jj}| \\
&\Longleftrightarrow P_{ij} = \max(P_{ii},P_{jj}) - 1 .
\end{aligned}
\]

\noindent\textbf{Case 2:} $ij\notin E(G)$. Then $M_{ij}\le 0$, so $|M_{ij}|=-M_{ij}$ and
$1+M_{ij}=|P_{ii}-P_{jj}|$. Hence
\[
\begin{aligned}
M_{ij}\le 0
&\Longleftrightarrow 1+M_{ij} = |P_{ii}-P_{jj}| \\
&\Longleftrightarrow 1+(P_{ii}+P_{jj}-2P_{ij}-1) = |P_{ii}-P_{jj}| \\
&\Longleftrightarrow P_{ij} = \min(P_{ii},P_{jj}) .
\end{aligned}
\]
Therefore, we have\[ij\in E(G)\;\Rightarrow
P_{ij}=\max(P_{ii},P_{jj})-1,
\]
\[ij\notin E(G)\;\Rightarrow
P_{ij}=\min(P_{ii},P_{jj}).
\]
Since $P$ is an orthogonal projection matrix, $0\leq P_{ii}\leq 1$. Suppose to the contrary that $P_{ii}=1$. From $P\mathbf{1}=\mathbf 0$ we have $\sum_j P_{ij}=0$, hence $1+\sum_{j\neq i} P_{ij}=0$. But for $j\neq i$, either $P_{ij}=\max(1,P_{jj})-1=0$ or $P_{ij}=\min(1,P_{jj})\geq 0$, so $1+\sum_{j\neq i} P_{ij}\geq 1$, which is a contradiction. Hence for $i\in [n]$, $0\leq P_{ii}<1$. 

Therefore, we have\[ij\in E(G)\;\Rightarrow
P_{ij}=\max(P_{ii},P_{jj})-1\;\Rightarrow P_{ij}<0,
\]
\[ij\notin E(G)\;\Rightarrow
P_{ij}=\min(P_{ii},P_{jj})\;\Rightarrow P_{ij}\geq0.
\]
Take the contrapositive of the two statements above and combine them we have
\[P_{ij}\geq0\;\Rightarrow P_{ij}\neq \max(P_{ii},P_{jj})-1\;\Rightarrow
ij\notin E(G)\;\Rightarrow
P_{ij}=\min(P_{ii},P_{jj})\;\Rightarrow P_{ij}\geq0,
\]
\[P_{ij}<0\;\Rightarrow P_{ij}\neq\min(P_{ii},P_{jj})\;\Rightarrow
ij\in E(G)\;\Rightarrow
P_{ij}=\max(P_{ii},P_{jj})-1\;\Rightarrow P_{ij}<0.
\]Thus we have
\[ij\in E(G)\;\Leftrightarrow P_{ij}=\max(P_{ii},P_{jj})-1\;\Leftrightarrow
P_{ij}<0,
\]
\[ij\notin E(G)\;\Leftrightarrow P_{ij}=\min(P_{ii},P_{jj})\;\Leftrightarrow
P_{ij}\geq0.
\]
Moreover, for $i\in [n]$, $P_{ii}=0\;\Rightarrow\;e_i^TPe_i=0\;\Rightarrow\;e_i^TP^TPe_i=0\Rightarrow\;\|Pe_i\|^2=0\;\Rightarrow\;Pe_i=\mathbf 0\;\Rightarrow\;e_j^TPe_i=P_{ij}=0$. Hence every vertex $j$ is not adjacent to $i$, and therefore $i$ is an isolated vertex.
\end{proof}

We may consider only graphs without isolated vertices; then $0<P_{ii}<1$ holds for every $i$.

\begin{definition}
Let $P=(P_{ij})$ be an $n\times n$ orthogonal projection matrix. Let $G$ be a graph of order $n$, $V(G)=[n]$. $(P,G)$ is called a {\it good pair} if for every $i$, $P_{ii}\in(0,1)$; and for every $i\neq j$,
$$ij\in E(G)\;\Longleftrightarrow\; P_{ij}=\max(P_{ii},P_{jj})-1,$$
$$ij\notin E(G) \;\Longleftrightarrow\; P_{ij}=\min(P_{ii},P_{jj}).$$
\end{definition}

\begin{lemma}\label{lem:001}
Let $(P,G)$ be a good pair. If for distinct vertices $i,j,k$,$$P_{ii}\geq P_{jj}\geq P_{kk},\;\ ij\notin E(G),\;\ ik\notin E(G),$$then $jk\notin E(G)$.
\end{lemma}
\begin{proof}
Suppose the contrary that $jk\in E(G)$. Consider the principal submatrix of $P$ indexed by $i,j,k$. Since\[
1>P_{ii}\geq P_{jj}\geq P_{kk}>0.
\]
By the definition of a good pair,
\[
P_{ij}=P_{jj},\qquad P_{ik}=P_{kk},\qquad P_{jk}=P_{jj}-1.
\]
The corresponding $3\times 3$ principal submatrix is
\[
A=
\begin{pmatrix}
P_{ii} & P_{jj} & P_{kk}\\
P_{jj} & P_{jj} & P_{jj}-1\\
P_{kk} & P_{jj}-1 & P_{kk}
\end{pmatrix}.
\]
Since $P$ is an orthogonal projection, it is positive semidefinite, so every principal minor is nonnegative; in particular $\det A\geq 0$.
We compute $\det A$:
\[
\begin{aligned}
\det A
&= P_{ii}\bigl(P_{jj}P_{kk}-(P_{jj}-1)^2\bigr) - P_{jj}P_{kk} + P_{jj}P_{kk}(P_{jj}-1-P_{kk}).
\end{aligned}
\]
Introduce
\[
x=P_{ii}-P_{jj},\qquad y=P_{jj}-P_{kk},\qquad z=1-P_{jj},
\]
so that
\[
0\leq x<z,\qquad 0\leq y<1-z,\qquad z>0.
\]
Substituting $P_{ii}=P_{jj}+x=1-z+x$ and $P_{kk}=P_{jj}-y=1-z-y$ gives
\[
\det A = x(1-y-2z+yz) - (1-y)^2(1-z).
\]
Notice that $1-y-2z+yz = (1-y)(1-z)-z$.

If $1-y-2z+yz\leq 0$, then $x\geq0$ implies
\[
\det A \leq -(1-y)^2(1-z) < 0,
\]
a contradiction.

If $1-y-2z+yz > 0$, then using $x<z$ and $y<1-z$ we obtain
\[
\begin{aligned}
\det A &< z(1-y-2z+yz) - (1-y)^2(1-z)
\\&=(1-z)(1-y)(z+y-1)-z^2
\\&<0,
\end{aligned}
\]
again a contradiction.

In every case $\det A<0$, which contradicts the positive semidefiniteness of $P$. Hence $j$ and $k$ cannot be adjacent.
\end{proof}

\begin{lemma}
If $(P,G)$ is a good pair, then $(I-P,\overline{G})$ is also a good pair.
\end{lemma}
\begin{proof}
Let $Q=(Q_{ij})=I-P$; then $Q_{ii}=1-P_{ii}\in(0,1)$ and $Q_{ij}=-P_{ij}$ for $i\neq j$. Clearly $Q$ is also an orthogonal projection. If $ij\in E(G)$, i.e., $ij\notin E(\overline{G})$, then
\[
Q_{ij} = -P_{ij} = -\max(P_{ii},P_{jj})+1 = \min(1-P_{ii},1-P_{jj}) = \min(Q_{ii},Q_{jj}).
\]
If $ij\notin E(G)$, i.e., $ij\in E(\overline{G})$, then
\[
Q_{ij} = -P_{ij} = -\min(P_{ii},P_{jj}) = \max(1-P_{ii},1-P_{jj})-1 = \max(Q_{ii},Q_{jj})-1.
\]
Thus $(I-P,\overline{G})$ is a good pair.
\end{proof}

\begin{lemma}\label{lem:011}
Let $(P,G)$ be a good pair. If for distinct vertices $i,j,k$,
\[
P_{ii}\geq P_{jj}\geq P_{kk},\;\ ij\notin E(G),\;\ ik\in E(G),
\]
then $jk\notin E(G)$.
\end{lemma}
\begin{proof}
Assume the contrary: $jk\in E(G)$. Set $Q=(Q_{ij}) = I-P$.
By the previous lemma, $(Q,\overline{G})$ is a good pair.

From $P_{ii}\geq P_{jj}\geq P_{kk}$ we obtain
\[
Q_{kk} \geq Q_{jj} \geq Q_{ii}.
\]
In $\overline{G}$ the edges are:
\[
jk\notin E(\overline{G}),\quad ik\notin E(\overline{G}),\quad ij\in E(\overline{G}).
\]
Reordering as $k,j,i$, this becomes
\[
kj\notin E(\overline{G}),\quad ki\notin E(\overline{G}),\quad ji\in E(\overline{G}),
\]
which contradicts Lemma~\ref{lem:001} applied to the good pair $(Q,\overline{G})$.
\end{proof}

Combining Lemma~\ref{lem:001} and Lemma~\ref{lem:011} yields the following.

\begin{corollary}\label{lem:0*1}
Let $(P,G)$ be a good pair. If for distinct vertices $i,j,k$,
\[
P_{ii}\geq P_{jj}\geq P_{kk},\;\ ij\notin E(G),
\]
then $jk\notin E(G)$.
\end{corollary}

\begin{theorem}\label{split}
If $(P,G)$ is a good pair, then $G$ is a split graph.
\end{theorem}
\begin{proof}
Relabel the vertices so that
\[
P_{11} \geq P_{22} \geq \cdots \geq P_{nn}.
\]
If $G$ has no edges, then $G$ itself is an independent set, hence a split graph.

Assume now that $G$ has at least one edge. Define
\[
r = \max\{\, s : \text{there exists } t>s \text{ with } st\in E(G) \,\}.
\]
In words, $r$ is the largest value among the smaller endpoints of all edges.

By definition, there exists $t>r$ such that $rt\in E(G)$. Set
\[
C = \{1,2,\dots,r\},\qquad I = \{r+1,r+2,\dots,n\}.
\]
We prove that $C$ is a clique and $I$ is an independent set.

First, $I$ is independent. If not, there exist $r<a<b\leq n$ with $ab\in E(G)$. Then $a$ is the smaller endpoint of edge $ab$ and $a>r$, contradicting the maximality of $r$.

Next, $C$ is a clique. Take arbitrary $1\leq a<b\leq r$. Since there exists $t>r$ with $rt\in E(G)$ and $a<r<t$, Lemma~\ref{lem:0*1} implies $ar\in E(G)$. In particular, for any $b<r$, we also have $br\in E(G)$. If $b=r$, then we already have $ar\in E(G)$, i.e., $ab\in E(G)$. If $b<r$, then from $br\in E(G)$ and $a<b<r$, Lemma~\ref{lem:0*1} again gives $ab\in E(G)$. Hence every pair $1\leq a<b\leq r$ satisfies $ab\in E(G)$, so $C$ is a clique.

Thus $V(G) = C \sqcup I$ where $C$ is a clique and $I$ is an independent set. Therefore $G$ is a split graph.
\end{proof}
\begin{corollary}\label{Gsplit}
    Let $G$ be a graph of order $n$. If there exists $1\leq k\leq n-1$ such that $s_k(G)=B_k(G),$ then $G$ is a split graph.
\end{corollary}
\begin{proof}
Let \(V(G)=[n]\). By Theorem \ref{reproduce}, if $s_k(G)=B_k(G)$, then there exists an $n\times n$ orthogonal projection matrix $P=(P_{ij})$ of rank $k$ with $P\mathbf 1=\mathbf 0$, such that
\[\sum_{ij\in E(G)}(P_{ii}+P_{jj}-2P_{ij}-1)=\binom{k+1}{2}.\]Then by Lemma \ref{Pii}, after ignoring isolated vertices, \((P,G)\) becomes a good pair. By
Theorem \ref{split}, the non-isolated part of \(G\) is split, and
adding isolated vertices preserves being split. Hence \(G\) is a split
graph.
\end{proof}
\section{Characterization of extremal split graphs}
\subsection{A refined lemma for split graphs}
\begin{lemma}[An improved form of Lemma 6 of \cite{Bai2011}]
\label{lem:improved-bai-lemma6}
Let \(G\) be a split graph with a fixed split partition
\[
V(G)=C\sqcup I,
\]
where \(C\) is a clique of size \(N\), and
\(I\) is an independent set of size $M$. If \(\lambda_N(G)\ge N\), then
\[
s_N(G)\le D_N(G).
\]
Moreover, if equality holds, then the family of neighbourhoods \(
\{N_C(y):y\in I\}
\) is linearly ordered by inclusion; that is, for any two vertices
\(y_1,y_2\in I\), either \(N_C(y_1)\subseteq N_C(y_2)\) or
\(N_C(y_2)\subseteq N_C(y_1)\).
\end{lemma}

\begin{proof}
We use the notation and the homotopy argument of \cite{Bai2011}. Denote by \(C=\{x_1,\ldots,x_N\}\) the clique of size \(N\), and
\(I=\{y_1,\ldots,y_M\}\) the independent set of size $M$. Let \(A=(a_{ij})\) be the \(N\times M\) bipartite adjacency matrix between
\(C\) and \(I\), where
\[
a_{ij}=1
\quad\Longleftrightarrow\quad
x_i y_j\in E(G).
\]
Let
\[
D_C=\operatorname{diag}(d_1,\ldots,d_N),
\qquad
D_I=\operatorname{diag}(f_1,\ldots,f_M),
\]
where
\[
d_i=|N_I(x_i)|,
\qquad
f_j=|N_C(y_j)|.
\]
Thus \(d_i\) is the number of neighbours of \(x_i\) in \(I\), and \(f_j\) is
the degree of \(y_j\). With this notation,
\[
L(G)=
\begin{pmatrix}
K_N+D_C & -A\\
-A^T & D_I
\end{pmatrix},
\]
where \(K_N\) is the Laplacian matrix of the complete graph on \(N\) vertices.
Put
\[
\Delta_I:=\max_{1\le j\le M} f_j\leq N.
\]
We first show that the hypothesis \(\lambda_N(G)\ge N\) implies the spectral
hypothesis used in Bai's Lemma 6. In other words, we first show that  \[
\lambda_N(G)\ge N\;\Longleftrightarrow \text{Either $\lambda_N(G)>N$ or $\lambda_N(G)=N>\Delta_I$}.
\]We remains to prove that
\[
\lambda_N(G)=N\;\Rightarrow \Delta_I<N.
\]
Assume to the contrary that \(\lambda_N(G)=N=\Delta_I\), then there exists
\(y\in I\) adjacent to every vertex of \(C\). Hence
\[
S:=C\cup\{y\}
\]
induces a complete graph \(K_{N+1}\). The principal submatrix \(L(G)[S]\)
satisfies
\[
L(G)[S]=L(K_{N+1})+\operatorname{diag}(r_v:v\in S),
\]
where
\[
r_v=d_G(v)-d_{G[S]}(v)\ge0.
\]
Therefore $\operatorname{diag}(r_v:v\in S)$ is semi-definite. By the Courant--Fischer min-max principle, the $N$-th largest eigenvalue
\[
\lambda_N(L(G)[S])\ge \lambda_N(K_{N+1})=N+1.
\]
By Cauchy interlacing for principal submatrices,
\[
\lambda_N(G)\ge \lambda_N(L(G)[S])\ge N+1,
\]
contradicting \(\lambda_N(G)=N\). Thus the assumptions of Bai's Lemma 6 are satisfied. Consequently,
\[
s_N(G)\le D_N (G).
\]

It remains to extract the equality information from Bai's proof. Under the
above spectral hypothesis, Bai's Lemma 8 gives the spectral separation needed
for the homotopy method. Bai's Lemmas 9--11 then yield a matrix
\[
V=(v_{ji})\in\mathbb R^{M\times N}
\]
corresponding to the invariant subspace spanned by the first \(N\) eigenvectors
of \(L(G)\), with the following properties:
\begin{equation}\label{vjidef}
v_{ji}\le0
\quad\text{for all }j,i,
\qquad
\sum_{j=1}^M v_{ji}=-1
\quad\text{for every }i,
\end{equation}
and
\begin{equation}\label{vjirelation}
v_{ji}(N+d_i)=-a_{ij}+f_jv_{ji}-\sum_{p=1}^N\sum_{q=1}^Mv_{jp}(1-a_{pq})v_{qi}\;\ \text{for all $j,i.$}
\end{equation}
Moreover, Bai's trace computation gives
\begin{equation}\label{trace}
\sum_{i=1}^N\lambda_i(G)
=
\operatorname{tr}(K_N+D_C-AV)
=
N(N-1)+e(C,I)-\operatorname{tr}(AV).
\end{equation}
Here \(e(C,I)=\sum_i d_i=\sum_j f_j\).

We also need the following consequence of \(\lambda_N(G)\ge N\). We claim that every vertex
of \(C\) has degree at least \(N\). Suppose to the contrary some \(x_i\in C\) had degree less
than \(N\), then \(x_i\) would have no neighbour in \(I\), and hence \(G\) could
also be viewed as a split graph with clique \(C\setminus\{x_i\}\) of size
\(N-1\). Applying Bai's Proposition 5 to this new split partition gives
\[
\lambda_N(G)=\lambda_{(N-1)+1}(G)\le N-1,
\]
contradicting \(\lambda_N(G)\ge N\). Therefore
\[
d_G(x_i)\ge N
\quad\text{for all }x_i\in C.
\]
Since every vertex of \(I\) has degree at most \(N\), we have
\begin{equation}\label{d^*}
\sum_{i=1}^N d_i^*(G)
=
\sum_{v\in V(G)}\min\{d_G(v),N\}
=
N^2+e(C,I).
\end{equation}
Now assume that equality holds:
\[
\sum_{i=1}^N\lambda_i(G)=\sum_{i=1}^N d_i^*(G).
\]
Combining (\ref{trace}) and (\ref{d^*}), we obtain
\begin{equation}\label{AV}
-\operatorname{tr}(AV)=N.
\end{equation}
On the other hand, by (\ref{vjidef}),
\[
\operatorname{tr}(AV)
=
\sum_{i=1}^N\sum_{j=1}^M a_{ij}v_{ji}
\ge
\sum_{i=1}^N\sum_{j=1}^M v_{ji}
=
-N.
\]
Hence equality in (\ref{AV}) forces
\begin{equation}\label{vjineed}
a_{ij}=0
\quad\Longrightarrow\quad
v_{ji}=0.
\end{equation}
We now prove that the family of neighbourhoods \(
\{N_C(y):y\in I\}
\) is linearly ordered by inclusion. Suppose
not. Then there exist \(y_{j_1},y_{j_2}\in I\) whose neighbourhoods in \(C\) are
incomparable. Hence we may choose vertices \(x_{i_1},x_{i_2}\in C\) such that
\[
x_{i_1}y_{j_1}\in E(G),\;\ x_{i_2}y_{j_2}\in E(G),\;\ x_{i_1}y_{j_2}\notin E(G),\;\ x_{i_2}y_{j_1}\notin E(G).
\]
Equivalently,
\[
a_{i_1j_1}=1,\quad a_{i_2j_2}=1,\quad a_{i_1j_2}=0,\quad a_{i_2j_1}=0.
\]
By (\ref{vjineed}), \(a_{i_2j_1}=0\) implies \(
v_{j_1i_2}=0
\). Taking $j=j_1,\ i=i_2$ in (\ref{vjirelation}), and using \(a_{i_2j_1}=0\), \(v_{j_1i_2}=0\), we
get
\begin{equation}\label{i2j1}
0=-\sum_{p=1}^N\sum_{q=1}^Mv_{j_1p}(1-a_{pq})v_{qi_2}.
\end{equation}
By (\ref{vjidef}), each summand is non-negative. Hence (\ref{i2j1}) forces
\begin{equation}\label{nonedge}
    v_{j_1p}(1-a_{pq})v_{qi_2}=0\;\ \text{for all $p,\ q$.}
\end{equation}
Taking $p=i_1,\ q=j_2$ in (\ref{nonedge}) and using $a_{i_1j_2}=0$ yields
\begin{equation}\label{youbian}
    v_{j_1i_1}v_{j_2i_2}=0.
\end{equation}
Hence $v_{j_1i_1}=0$ or $v_{j_2i_2}=0$. Without loss of generality, assume $v_{j_1i_1}=0$. Taking $j=j_1,\ i=i_1$ in (\ref{vjirelation}), we get
\[0=-1-\sum_{p=1}^N\sum_{q=1}^Mv_{j_1p}(1-a_{pq})v_{qi_1}\leq -1,
\]which is a contradiction. Therefore no two vertices of \(I\) have incomparable
neighbourhoods in \(C\). Hence \(
\{N_C(y):y\in I\}
\) is linearly ordered by inclusion.
\end{proof}
\subsection{Equality case for split graphs}
\begin{theorem}
\label{Case when k=N}
Let \(G\) be a split graph of order $n$. Choose a split partition \(V(G)=C\sqcup I\), where \(C\) is a clique, \(I\) is a maximal independent set. Put \(N:=|C|\).
Then
\[s_N(G)=B_N(G)\]
if and only if \(G\) is a threshold graph with clique number \(N+1\).
Equivalently, \(G\in \mathcal{G}_{N,r,s}\) for some \(r\geq 1\) and \(s\geq 0\).
\end{theorem}
\begin{proof}
If $N=0$, which means $e(G)=0$, then $s_0(G)=B_0(G)=0$ and $G$ is a threshold graph with clique number $1$. So assume $N\geq 1$.

By Lemma \ref{lem:berndsen-split} and the Grone-Merris-Bai inequality,
$$s_N(G)=D_N(G)=B_N(G),$$
$$s_{N-1}(G)\leq D_{N-1}(G).$$
Since all $N$ vertices of $C$ have degree at least $N$, taking the difference of them yields $\lambda_N(G)\geq d_N^*(G)\geq N.$ Hence by Lemma \ref{lem:improved-bai-lemma6}, the family of neighbourhoods \(
\{N_C(y):y\in I\}
\) is linearly ordered by inclusion, which means $G$ is a threshold graph. It remains to prove that $G$ is of clique number $N+1$, i.e., there exists $v\in I$ such that $e(v,C)=N$. Since $\{N_C(y):y\in I\}$ is linearly ordered by inclusion, set $I=\{v_1,\ldots,v_{n-N}\}$ such that $$N_C(v_1)\supseteq N_C(v_2)\cdots\supseteq N_C(v_{n-N}).$$
By the maximality of $I$, for every $x\in C$, there exists $i$ such that $x\in N_C(v_i)\subseteq N_C(v_1)$. Hence $e(v_1, C)=N$, which completes the proof.
\end{proof}
\begin{theorem}
\label{Case when knot N}
Let \(G\) be a split graph of order $n$. Choose a split partition \(V(G)=C\sqcup I\), where \(C\) is a clique, \(I\) is a maximal independent set. Put \(N:=|C|\).
Then for $1\leq k\leq n-1$ and $k\neq N$,
\[s_k(G)<B_k(G).\]
\end{theorem}
\begin{proof}
Assume, to the contrary, that $s_k(G)=B_k(G)$ for some $k\neq N,\ 1\leq k\leq n-1$. By Lemma \ref{lem:berndsen-split} and the Grone-Merris-Bai inequality, $$s_k(G)=D_k(G)=B_k(G).$$Thus $f_k(G)=0$. By Lemma \ref{lem:zeros-of-f}, $k=N-1$. By Lemma \ref{lem:complement-GMB},
$$s_{n-N}(\overline G)=D_{n-N}(\overline G).$$
By Theorem \ref{Case when k=N}, $\overline G$ is a threshold graph of clique number $n-N+1$, which means there exists $x\in C$ such that $e_{\overline G}(x,I)=n-N$. That is, $e_G(x,I)=0$, contradicting to the maximality of $I$.
\end{proof}
\section{Proof of the full Brouwer conjecture}
\begin{pot}
Suppose first that equality holds. By Corollary \ref{Gsplit}, \(G\) is a split
graph. Then by Theorems \ref{Case when k=N} and \ref{Case when knot N}, $G\in \mathcal{G}_{k,r,s}$. In other words, $G$ is a threshold graph of clique number $k+1$.

Conversely, if \(G\) is a threshold graph with clique number \(k+1\), then
by Theorem \ref{thm:LiGuo-split}, $s_k(G)=B_k(G)$. This proves the equality condition.
\end{pot}
\subsection*{Acknowledgements}
This work is partly supported by the National Natural Science Foundation of China (No.12371354, W2521102), the Montenegrin-Chinese Science and Technology Cooperation Project (No.4-3)  and  the Science and Technology Commission of Shanghai Municipality (No.25LN3200600).

\bibliographystyle{cas-name}

\bibliography{cas-refs}
\end{document}